\documentclass[11pt, final]{article}
\usepackage{a4}
\usepackage{amsmath}%
\usepackage{amstext}%
\usepackage{amssymb}%
\usepackage{showkeys}%
\usepackage{epsfig}%
\usepackage{cite}
\newtheorem{theorem}{Theorem}

\newtheorem{axiom}{Axiom}

\newtheorem{corollary}[theorem]{Corollary}

\newtheorem{definition}[axiom]{Definition}
\newtheorem{example}[theorem]{Example}

\newtheorem{lemma}[theorem]{Lemma}

\newtheorem{proposition}[theorem]{Proposition}
\newtheorem{remark}{Remark}

\newcounter{unnumber}

\newtheorem{prf}[unnumber]{Proof.}
\newenvironment{proof}{\prf\rm}{\hfill{$\blacksquare$}\endprf}

\newcommand{\R}{\mathbb{R}}%
\newcommand{\E}{\mathbb{E}}%

\DeclareMathOperator*\dom{dom}%
\DeclareMathOperator*\epi{epi}%
\DeclareMathOperator*\co{co}%
\DeclareMathOperator*\cl{cl}%
\DeclareMathOperator*\B{\overline{\R}}%
\DeclareMathOperator*\lev{le}%

\title{Enlargements of positive sets}

 \author{Radu Ioan Bo\c{t} \thanks
 {Faculty of Mathematics, Chemnitz University of Technology,
D-09107 Chemnitz, Germany, e-mail:
 radu.bot@mathematik.tu-chemnitz.de. Research partially supported by DFG (German Research Foundation), project WA 922/1-3.} \and Ern\"{o} Robert Csetnek
 \thanks {Faculty of Mathematics, Chemnitz University of Technology,
D-09107 Chemnitz, Germany, e-mail:
 robert.csetnek@mathematik.tu-chemnitz.de. Research supported by a Graduate Fellowship of the Free State Saxony, Germany.}}

\begin{document}
\maketitle

\noindent \textbf{Abstract.} In this paper we introduce the notion
of enlargement of a positive set in SSD spaces. To a maximally
positive set $A$ we associate a family of enlargements $\E(A)$ and
characterize the smallest and biggest element in this family with
respect to the inclusion relation. We also emphasize the existence
of a bijection between the subfamily of closed enlargements of
$\E(A)$ and the family of so-called representative functions of
$A$. We show that the extremal elements of the latter family are
two functions recently introduced and studied by Stephen Simons.
In this way we extend to SSD spaces some former results given for
monotone and maximally monotone sets in Banach spaces.\\



\section{Introduction}

The notion of positive set with respect to a quadratic form
defined on a so-called \emph{symmetrically self-dual space Banach
space (Banach SSD space)} has been introduced by Stephen Simons in
\cite{simons-ssd-jca} (see also \cite{simons}) as an extension of
a monotone set in a Banach space. A number of known results coming
from the theory of monotone operators have been successfully
generalized to this framework. In his investigations Simons has
mainly used some techniques based on the extension of the notion
of Fitzpatrick function from the theory of monotone sets to a
similar concept for positive sets. These investigations have been
continued by the same author in \cite{simons-ssd}, where notions
and results recently introduced in the theory of monotone
operators in general Banach spaces have known an appropriate
generalization to positive sets in Banach SSD spaces.

In analogy to the enlargement of a monotone operator we introduce
and study in this article the notion of enlargement of a positive
set in a SSD space. In this way we extend to SSD spaces some
results given in the literature for monotone and maximally
monotone sets.

In Section 2 we recall the definition of a SSD space $B$ along
with some examples given in \cite{simons-ssd} and give a calculus
rule for the quadratic form $q :B \rightarrow \R$  considered on
it.

In Section 3 we introduce the notion of enlargement of a
$q$-positive set (positive with respect to $q$) $A \subseteq B$ as
being the multifunction $E:\R_+ \rightrightarrows B$ fulfilling $A
\subseteq E(\varepsilon)$ for all $\varepsilon \geq 0$. In
connection with this notion we introduce the so-called
\emph{transportation formula} and provide a characterization of
enlargements which fulfills it. We also also associate to $A$ a
family of enlargements $\E(A)$ for which we provide, in case $A$
is $q$-maximally positive, the smallest and the biggest element
with respect to the partial ordering inclusion relation of the
graphs. In this way we extend to SSD spaces some former results
given in \cite{bu-sv-99, sv, bu-sv-02}.

In Section 4 we assume that $B$ is a Banach SSD Space and deal
with $\E_c(A)$, a subfamily of $\E(A)$, containing those
enlargements of $\E(A)$ having a closed graph. For $\E_c(A)$ we
point out, in case $A$ is $q$-maximally monotone, the smallest and
the biggest element with respect to the partial ordering inclusion
relation of the graphs, too, as well as a bijection between this
subfamily and the set of so-called \emph{representative functions
of $A$}
$${\cal H}(A) = \{h:B \rightarrow \overline{\R}: h \ \mbox{convex, lower semicontinuous}, \
h \geq q \ \mbox{on} \ B, h=q \ \mbox{on} \ A\}.$$ These results
generalize to Banach SSD spaces the ones given in \cite{sv,
bu-sv-02} for enlargements of maximally monotone operators. We
also show that the smallest and the biggest element of ${\cal
H}(A)$ are nothing else than the functions $\Phi_A$ and $^*
\Theta_A$ considered in \cite{simons-ssd} and provide some
characterizations of these functions beyond the ones given in the
mentioned paper. We close the paper by giving a characterization
of the additive enlargements in $\E_c(A)$, in case $A$ is a
$q$-maximally monotone set, which turns out to be helpful when
showing the existence of enlargements having this property.

\section{Preliminary notions and results}

Consider $X$ a real separated locally convex space and $X^*$ its
topological dual space. The notation $\omega(X^*,X)$ stands for
the weak$^*$ topology induced by $X$ on $X^*$, while by $\langle
x,x^*\rangle$ we denote the value of the linear continuous
functional $x^*\in X^*$ at $x\in X$. For a subset $C$ of $X$ we
denote by $\overline{C}$ and $\co(C)$ its \emph{closure} and
\emph{convex hull}, respectively. We also consider the
\emph{indicator function} of the set $C$, denoted by $\delta_C$,
which is zero for $x \in C$ and $+\infty$ otherwise.

For a function $f:X\rightarrow\B=\R\cup\{\pm\infty\}$ we denote by
$\dom(f)=\{x\in X:f(x)<+\infty\}$ its \emph{domain} and by
$\epi(f)=\{(x,r)\in X\times\R:f(x)\leq r\}$ its \emph{epigraph}.
We call $f$ \emph{proper} if $\dom(f)\neq\emptyset$ and
$f(x)>-\infty$ for all $x\in X$. By $\cl(f)$ we denote the
\emph{lower semicontinuous hull} of $f$, namely the function of
which epigraph is the closure of $\epi(f)$ in $X\times\R$, that is
$\epi(\cl(f))=\cl(\epi(f)).$ The function $\co(f)$ is the largest
convex function majorized by $f$.

Having $f:X\rightarrow\B$ a proper function, for $x\in\dom(f)$ we
define the \emph{$\varepsilon$-subdifferential} of $f$ at $x$,
where $\varepsilon\geq 0$, by
$$\partial_{\varepsilon} f(x)=\{x^*\in X^*:f(y)-f(x)\geq\langle
y-x,x^*\rangle-\varepsilon\mbox{ for all }y\in X\}.$$ For
$x\notin\dom(f)$ we take $\partial_{\varepsilon} f(x):=\emptyset$.
The set $\partial f(x):=\partial_0 f(x)$ is then the classical
\emph{subdifferential} of $f$ at $x$.

The \emph{Fenchel-Moreau conjugate} of $f$ is the function
$f^*:X^*\rightarrow\B$ defined by $f^*(x^*)=\sup_{x\in X}\{\langle
x,x^*\rangle-f(x)\}\mbox{ for all }x^*\in X^*$. Next we mention
some important properties of conjugate functions. We have the
so-called \emph{Young-Fenchel inequality}
$f^*(x^*)+f(x)\geq\langle x,x^*\rangle\mbox{ for all }(x,x^*)\in
X\times X^*.$ If $f$ is proper, then $f$ is convex and lower
semicontinuous if and only if $f^{**}=f$ (see \cite{EkTem,
Zalcarte}). As a consequence we have that in case $f$ is convex
and $\cl(f)$ is proper, then $f^{**}=\cl(f)$ (cf. \cite[Theorem
2.3.4]{Zalcarte}).

One can give the following characterizations for the
subdifferential and $\varepsilon$-subdifferential of a proper
function $f$ by means of conjugate functions (see \cite{EkTem,
Zalcarte}):
$$x^*\in\partial f(x)\Leftrightarrow f(x)+f^*(x^*)=\langle
x,x^*\rangle$$ and, respectively,
$$x^*\in\partial_{\varepsilon}f(x)\Leftrightarrow
f(x)+f^*(x^*)\leq\langle x,x^*\rangle+\varepsilon.$$

\begin{definition}\label{ssd} (cf. \cite[Definition 1.2]{simons-ssd}) (i) We say that $(B, \lfloor \cdot,\cdot \rfloor)$
is a \textbf{symmetrically self-dual space (SSD space)} if $B$ is
a nonzero real vector space and $\lfloor \cdot,\cdot
\rfloor:B\times B\rightarrow \R$ is a symmetric bilinear form. We
consider the quadratic form $q:B\rightarrow\R$,
$q(b)=\frac{1}{2}\lfloor b,b \rfloor$ for all $b\in B$.

(ii) A subset $A$ of $B$ is said to be \textbf{$q$-positive} if
$A\neq\emptyset$ and $q(b-c)\geq 0$ for all $b,c\in A$. We say
that $A$ is \textbf{maximally $q$-positive} if $A$ is $q$-positive
and maximal (with respect to the inclusion) in the family of
$q$-positive subsets of $B$.
\end{definition}

\begin{remark}\label{q-max} (i) In every SSD space $B$ the following calculus rule is fulfilled: \begin{equation}\label{q}q(\alpha b+\gamma
c)=\alpha^2q(b)+\gamma^2q(c)+\alpha\gamma\lfloor b,c \rfloor\mbox{
for all } \alpha,\gamma\in\R\mbox{ and }b,c\in B.\end{equation}

(ii) Let $B$ be an SSD space and $A\subseteq B$ be a $q$-positive
set. Then $A$ is maximally $q$-positive if and only if for all
$b\in B$ the implication below holds $$q(b-c)\geq 0\mbox{ for all
}c\in A\Rightarrow b\in A.$$
\end{remark}

\begin{example}\label{ex-ssd-sp} (cf. \cite{simons-ssd}) (a) If $B$ is a Hilbert space
with inner product $(b,c)\mapsto\langle b,c\rangle$ then $B$ is a
SSD space with $\lfloor b,c \rfloor=\langle b,c\rangle$ and
$q(b)=\frac{1}{2}\|b\|^2$ and every subset of $B$ is $q$-positive.

(b)If $B$ is a Hilbert space with inner product
$(b,c)\mapsto\langle b,c\rangle$ then $B$ is a SSD space with
$\lfloor b,c \rfloor=-\langle b,c\rangle$ and
$q(b)=-\frac{1}{2}\|b\|^2$ and the $q$-positive sets are the
singletons.

(c) $\R^3$ is a SSD space with $\lfloor
(b_1,b_2,b_3),(c_1,c_2,c_3) \rfloor=b_1c_2+b_2c_1+b_3c_3$ and
$q(b_1,b_2,b_3)=b_1b_2+\frac{1}{2}b_3^2$. See \cite{simons-ssd}
for a discussion regarding $q$-positive sets in this setting.

(d) Consider $X$ a nonzero Banach space and $B:=X\times X^*$. For
all $b=(x,x^*)$ and $c=(y,y^*)\in B$ we set $\lfloor
b,c\rfloor:=\langle x,y^*\rangle+\langle y,x^*\rangle$. Then $B$
is an SSD space, $q(b)=\langle x,x^*\rangle$ and $q(b-c)=\langle
x-y,x^*-y^*\rangle$. Hence for $A\subseteq B$ we have that $A$ is
$q$-positive exactly when $A$ is a nonempty monotone subset of
$X\times X^*$ in the usual sense and $A$ is maximally $q$-positive
exactly when $A$ is a maximally monotone subset of $X\times X^*$
in the usual sense. The classical example of a maximal monotone
operator is the subdifferential of a proper, convex and lower
semicontinuous function (see \cite{rock}). However, there exist
maximal monotone operators which are not subdifferentials (see
\cite{simons}).
\end{example}

Let us consider in the following an arbitrary SSD space $B$ and a
function $f:B\rightarrow \B$. We write $f^@$ for the conjugate of
$f$ with respect to the pairing $\lfloor\cdot,\cdot\rfloor$, that
is $f^@(c)=\sup_{b\in B}\{\lfloor b,c\rfloor-f(b)\}$. We write
${\cal P}(f):=\{b\in B:f(b)=q(b)\}$. If $f$ is proper, convex,
$f\geq q$ on $B$ and ${\cal P}(f)\neq\emptyset$, then ${\cal
P}(f)$ is a $q$-positive subset of $B$ (see \cite[Lemma
1.9]{simons-ssd}). Conditions under which ${\cal P}(f)$ is
maximally $q$-positive are given in \cite[Theorem
2.9]{simons-ssd}.

\section{Enlargements of positive sets in SSD spaces}

In this section we introduce the concept of enlargement of a
positive set and study some of its algebraic properties.

\begin{definition}\label{enlar-pos} Let $B$ be an SSD space. Given $A$ a $q$-positive subset of $B$, we say that the
multifunction $E:\R_+\rightrightarrows B$ is an
\textbf{enlargement} of $A$ if $$A\subseteq E(\varepsilon)\mbox{
for all }\varepsilon\geq 0.$$\end{definition}

\begin{example}\label{E_A} Let $B$ be an SSD space and $A$ a $q$-positive subset of
$B$. The multifunction $E^A:\R_+\rightrightarrows B$,
$$E^A(\varepsilon):= \{b\in B:q(b-c)\geq -\varepsilon\mbox{ for
all }c\in A\}$$ is an enlargement of $A$. Let us notice that $A$
is maximally $q$-positive if and only if $A=E^A(0)$. Moreover, in
the framework of Example \ref{ex-ssd-sp} (d), for the graph of
$E^A$ we have $G(E^A)=\{(\varepsilon,x,x^*):\langle
x-y,x^*-y^*\rangle\geq-\varepsilon\mbox{ for all }(y,y^*)\in A\}$,
hence in this case $G(E^A)=G(B^A)$, where $B^A:\R_+\times
X\rightrightarrows X^*$, $B^A(\varepsilon,x)=\{x^*:\langle
x-y,x^*-y^*\rangle\geq-\varepsilon,\forall (y,y^*)\in A\}$, was
defined in \cite{bur-ius-sv} and studied in \cite{bur-iu,
bur-iu-carte, bur-ius-sv, bu-sv-99, bu-sv-02, sv}.
\end{example}

\noindent The following definition extends to SSD spaces a notion
given in \cite{bur-sag-sv, bu-sv-99} (see also \cite[Definition
2.3]{bu-sv-02}).

\begin{definition}\label{transp-f2} We say that the multifunction
$E:\R_+\rightrightarrows B$ satisfies the \textbf{transportation
formula} if for every $\varepsilon_1, \varepsilon_2\geq 0$,
$b^1\in E(\varepsilon_1)$, $b^2\in E(\varepsilon_2)$ and every
$\alpha_1,\alpha_2\geq 0$, $\alpha_1+\alpha_2=1$ we have
$\varepsilon:=\alpha_1\varepsilon_1+\alpha_2\varepsilon_2+\alpha_1\alpha_2q(b^1-b^2)\geq
0$ and $\alpha_1b^1+\alpha_2b^2\in E(\varepsilon)$.
\end{definition}

\begin{proposition}\label{E_A-transp-f2} Let $B$ be an SSD space
and $A\subseteq B$ be a maximally $q$-positive set. Then $E^A$
satisfies the transportation formula.
\end{proposition}

\begin{proof} Take  $\varepsilon_1, \varepsilon_2\geq 0$,
$b^1\in E^A(\varepsilon_1)$, $b^2\in E^A(\varepsilon_2)$ and
$\alpha_1,\alpha_2\geq 0$, $\alpha_1+\alpha_2=1$. We have to show
that
\begin{equation}\label{ineq-tr-f-EA}q(\alpha_1b^1+\alpha_2b^2-c)\geq
-\alpha_1\varepsilon_1-\alpha_2\varepsilon_2-\alpha_1\alpha_2q(b^1-b^2)\mbox{
for all }c\in A\end{equation} and
\begin{equation}\label{ineq-tr-f-EA2}\alpha_1b^1+\alpha_2b^2+\alpha_1\alpha_2q(b^1-b^2)\geq
0.\end{equation} Let $c$ be an arbitrary element in $A$. By using
the inequalities $q(b^1-c)\geq-\varepsilon_1$ and
$q(b^2-c)\geq-\varepsilon_2$ and the calculus rule \eqref{q} we
obtain
$$q(\alpha_1b^1+\alpha_2b^2-c)=q\big(\alpha_1(b^1-c)+\alpha_2(b^2-c)\big)=\alpha_1^2q(b^1-c)+\alpha_2^2q(b^2-c)
+\alpha_1\alpha_2\lfloor
b^1-c,b^2-c\rfloor$$$$=\alpha_1^2q(b^1-c)+\alpha_2^2q(b^2-c)
+\alpha_1\alpha_2\big(q(b^1-c)+q(b^2-c)-q(b^1-b^2)\big)$$$$=\alpha_1q(b^1-c)+\alpha_2q(b^2-c)-\alpha_1\alpha_2q(b^1-b^2)
\geq
-\alpha_1\varepsilon_1-\alpha_2\varepsilon_2-\alpha_1\alpha_2q(b^1-b^2).$$
Let us suppose that
$\alpha_1b^1+\alpha_2b^2+\alpha_1\alpha_2q(b^1-b^2)< 0$. From
\eqref{ineq-tr-f-EA} we obtain
\begin{equation}\label{contr}q(\alpha_1b^1+\alpha_2b^2-c)> 0\mbox{ for all }c\in A.\end{equation} Since
$A$ is maximally $q$-positive we get $\alpha_1b^1+\alpha_2b^2\in
A$ (cf. Remark \ref{q-max} (ii)). By choosing
$c:=\alpha_1b^1+\alpha_2b^2\in A$ in \eqref{contr} we get $0>0$,
which is a contradiction. Hence \eqref{ineq-tr-f-EA2} is also
fulfilled and the proof is complete. \end{proof}

The following result establishes a connection between the
transportation formula and convexity (see also \cite[Lemma
3.2]{sv}).

\begin{proposition}\label{transp-f2n-conv} Let $B$ be an SSD
space, $E:\R_+\rightrightarrows B$ a multifunction and define the
function $\Psi:\R\times B\rightarrow \R\times B$,
$\Psi(\epsilon,b)=(\epsilon+q(b), b)$ for all $(\epsilon,b)\in
\R\times B$. The following statements are equivalent:
\begin{enumerate}\item[(i)] $E$ satisfies the transportation formula;

\item[(ii)] $E$ satisfies the \textbf{generalized transportation
formula} (or the \textbf{n-point transportation formula}), that is
for all $n\geq 1 $, $\varepsilon_i\geq 0$, $b^i\in
E(\varepsilon_i)$ and $\alpha_i\geq 0$, $i=1,...,n$, with
$\sum_{i=1}^{n}\alpha_i=1$ we have
$\varepsilon:=\sum_{i=1}^{n}\alpha_i\varepsilon_i+\sum_{i=1}^{n}\alpha_iq\big(b^i-\sum_{j=1}^{n}\alpha_jb^j\big)\geq
0$ and $\sum_{i=1}^{n}\alpha_ib^i\in E(\varepsilon)$;

\item[(iii)] $\Psi(G(E))$ is convex.
\end{enumerate}
\end{proposition}

\begin{proof} We notice first that $\Psi$ is a bijective function
with inverse $\Psi^{-1} :\R\times B\rightarrow \R\times B$,
$\Psi^{-1}(\epsilon,b) = (\epsilon-q(b),b)$.

(ii)$\Rightarrow$(i) Take  $\varepsilon_1, \varepsilon_2\geq 0$,
$b^1\in E(\varepsilon_1)$, $b^2\in E(\varepsilon_2)$ and
$\alpha_1,\alpha_2\geq 0$, $\alpha_1+\alpha_2=1$. Then
$\varepsilon:=\alpha_1 \varepsilon_1 + \alpha_2 \varepsilon_2 +
\alpha_1 q(b^1-(\alpha_1 b^1 + \alpha_2 b^2)) + \alpha_2
q(b^2-(\alpha_1 b^1 + \alpha_2 b^2)) \geq 0$ and $\alpha_1 b^1 +
\alpha_2 b^2 \in E(\varepsilon)$. Since $$\varepsilon=\alpha_1
\varepsilon_1 + \alpha_2 \varepsilon_2 + \alpha_1 \alpha_2^2
q(b^1-b^2) + \alpha_2 \alpha_1^2q(b^1-b^2) = \alpha_1
\varepsilon_1 + \alpha_2 \varepsilon_2 +
\alpha_1\alpha_2q(b^1-b^2),$$ this implies that $E$ satisfies the
transportation formula.

(i)$\Rightarrow$(iii) Let be $(\mu_1,b^1), (\mu_2,b^2) \in
\Psi(G(E))$ and $\alpha_1,\alpha_2\geq 0$ with
$\alpha_1+\alpha_2=1$. Then there exist $\varepsilon_1,
\varepsilon_2 \geq 0$ such that $\mu_1 = \varepsilon_1 + q(b^1)$,
$b^1 \in E(\varepsilon_1)$ and $\mu_2 = \varepsilon_2 + q(b^2)$,
$b^1 \in E(\varepsilon_2)$. By (i) we have that
$\varepsilon:=\alpha_1\varepsilon_1+\alpha_2\varepsilon_2+\alpha_1\alpha_2q(b^1-b^2)\geq
0$ and $\alpha_1b^1+\alpha_2b^2\in E(\varepsilon)$. Using
\eqref{q} we further get that
$$\varepsilon + q(\alpha_1 b^1 + \alpha_2
b^2) = \alpha_1 \varepsilon_1 + \alpha_2 \varepsilon_2 +
\alpha_1\alpha_2 (q(b^1) + q(b^2) - \lfloor b^1, b^2 \rfloor) +
\alpha_1^2 q(b^1) + \alpha_2^2 q(b^2) +$$ $$\alpha_1\alpha_2
\lfloor b^1, b^2 \rfloor = \alpha_1 \varepsilon_1 + \alpha_2
\varepsilon_2 + \alpha_1 q(b^1) + \alpha_2 q(b^2) = \alpha_1 \mu_1
+ \alpha_2 \mu_2.$$ Thus
$$\alpha_1(\mu_1,b^1) + \alpha_2(\mu_2,b^2) = (\varepsilon + q(\alpha_1 b^1 + \alpha_2
b^2), \alpha_1 b^1 + \alpha_2 b^2) = \Psi(\varepsilon, \alpha_1
b^1 + \alpha_2 b^2) \in \Psi(G(E))$$ and this provides the
convexity of $\Psi(G(E))$.

(iii)$\Rightarrow$(ii) Let be $n\geq 1 $, $\varepsilon_i\geq 0$,
$b^i\in E(\varepsilon_i)$ and $\alpha_i\geq 0$, $i=1,...,n$, with
$\sum_{i=1}^{n}\alpha_i=1$. This means that $(\varepsilon_i +
q(b^i), b^i) \in \Psi(G(E))$ for $i=1,...,n$. Denote
$b:=\sum_{i=1}^{n}\alpha_ib^i$ and $\varepsilon:= \sum_{i=1}^n
\alpha_i(\varepsilon_i + q(b^i)) - q(b)$. By using the convexity
of $\Psi(G(E))$ one has $(\varepsilon + q(b),b) = \sum_{i=1}^n
\alpha_i (\varepsilon_i + q(b^i), b^i) \in \Psi(G(E))$, which
implies that $\Psi^{-1}(\varepsilon + q(b),b) = (\varepsilon, b)
\in G(E)$. From here if follows that $\varepsilon = \sum_{i=1}^n
\alpha_i\varepsilon_i + \sum_{i=1}^n \alpha_i q(b^i) - q(b) \geq
0$ and $b = \sum_{i=1}^{n}\alpha_ib^i \in E(\varepsilon)$. To
conclude the proof we have only to show that $\sum_{i=1}^n
\alpha_i q(b^i - b) = \sum_{i=1}^n \alpha_i q(b^i) - q(b)$. Indeed
$$\sum_{i=1}^n \alpha_i q(b^i - b) = \sum_{i=1}^n \alpha_i q(b^i) + \sum_{i=1}^n \alpha_i q(b) - \sum_{i=1}^n \alpha_i \lfloor b^i, b \rfloor =
\sum_{i=1}^n \alpha_i q(b^i) + q(b) - \lfloor b,b \rfloor  =$$
$$\sum_{i=1}^n \alpha_i q(b^i) + q(b) - 2 q(b) = \sum_{i=1}^n
\alpha_i q(b^i) - q(b)$$ and, consequently, the generalized
transportation formula holds.\end{proof}

\noindent As in \cite[Definition 3.3]{sv} (see also
\cite[Definition 2.5]{bu-sv-02}) one can introduce a family of
enlargements associated to a positive set.

\begin{definition}\label{fam-enl} Let $B$ be an SSD space and $A\subseteq
B$ be a $q$-positive set. We define $\E(A)$ as being the family of
multifunctions $E:\R_+\rightrightarrows B$ satisfying the
following properties:

(r1) $E$ is an enlargement of $A$, that is $$A\subseteq
E(\varepsilon)\mbox{ for all }\varepsilon\geq 0;$$

(r2) $E$ is nondecreasing:
$$0\leq\varepsilon_1\leq\varepsilon_2\Rightarrow E(\varepsilon_1)\subseteq E(\varepsilon_2);$$

(r3) $E$ satisfies the transportation formula.
\end{definition}

If $A$ is maximally $q$-positive then $E^A$ satisfies the
properties (r1)-(r3) (cf. Example \ref{E_A} and Proposition
\ref{E_A-transp-f2}, while (r2) is obviously satisfied), hence in
this case the family $\E(A)$ is nonempty. Define the multifunction
$E_A:\R_+\rightrightarrows B$,
$E_A(\varepsilon):=\bigcap\nolimits_{E\in \E(A)}E(\varepsilon)$
for all $\varepsilon\geq 0$.

\begin{proposition}\label{sm-big-E} Let $B$ be an SSD space and $A\subseteq
B$ be a maximally $q$-positive set. Then \begin{enumerate}
\item[(i)] $E_A, E^A\in \E(A)$;

\item[(ii)] $E_A$ and $E^A$ are, respectively, the smallest and
the biggest elements in $\E(A)$ with respect to the partial
ordering inclusion relation of the graphs, that is
$G(E_A)\subseteq G(E)\subseteq G(E^A)$ for all $E\in \E(A)$.
\end{enumerate}
\end{proposition}

\begin{proof} (i) That $E^A\in
\E(A)$ was pointed out above. The statement $E_A\in\E(A)$ follows
immediately, if we take into consideration the definition of
$E_A$.

(ii) $E_A$ is obviously the smallest element in $\E(A)$. We prove
in the following that $E^A$ is the biggest element in $\E(A)$.
Suppose that  $E^A$ is not the biggest element in $\E(A)$, namely
that there exist $E\in\E(A)$ and $(\varepsilon,b)\in G(E)\setminus
G(E^A)$. Since $(\varepsilon,b)\not\in G(E^A)$, there exists $c\in
A$ such that $q(b-c)<-\varepsilon$. Let $\lambda\in(0,1)$ be
fixed. As $E$ satisfies (r1), we have $c\in A\subseteq E(0)$, that
is $(0,c)\in G(E)$. As $(\varepsilon,b),(0,c)\in G(E)$,
$\lambda\in (0,1)$ and $E$ satisfies the transportation formula,
we obtain $\lambda\varepsilon+\lambda(1-\lambda) q(b-c) \geq 0$,
hence $\varepsilon+(1-\lambda) q(b-c)\geq 0$. Since this
inequality must hold for arbitrary $\lambda\in (0,1)$ we get
$\varepsilon+q(b-c)\geq 0$, which is a contradiction.\end{proof}

\begin{lemma}\label{gen-prop-enl} Let $B$ be an SSD space, $A\subseteq
B$ a maximally $q$-positive set and $E\in \E(A)$. Then
$$E(0)=\bigcap_{\varepsilon>0}E(\varepsilon)=A.$$
\end{lemma}

\begin{proof} By using the properties (r1), (r2), Proposition
\ref{sm-big-E}, the definition of $E^A$ and Example \ref{E_A} we
get $$A\subseteq E(0)\subseteq
\bigcap_{\varepsilon>0}E(\varepsilon)\subseteq
\bigcap_{\varepsilon>0}E^A(\varepsilon)=E^A(0)=A$$ and the
conclusion follows. \end{proof}

\section{Enlargements of positive sets in Banach SSD spaces}

We start by recalling the definition of a Banach SSD space,
concept introduced by Stephen Simons, and further we study some
topological properties of enlargements of positive sets in this
framework. By following the suggestion of Heinz Bauschke we
propose calling this class of spaces \emph{Simons spaces} and hope
that Stephen agrees to this proposal.

\begin{definition}\label{banach-ssd} We say that $B$ is a \textbf{Banach SSD space (Simons space)} if
$B$ is an SSD space and $\|\cdot\|$ is a norm on $B$ with respect
to which $B$ is a Banach space with norm-dual $B^*$,
\begin{equation}\label{norm-q}\frac{1}{2}\|\cdot\|^2+q\geq 0\mbox{ on }B\end{equation}
and there exists $\iota:B\rightarrow B^*$ linear and continuous
such that \begin{equation}\label{iota} \langle
b,\iota(c)\rangle=\lfloor b,c \rfloor\mbox{ for all }b,c\in
B.\end{equation}\end{definition}

\begin{remark}\label{cont-q} (i) From \eqref{norm-q} we obtain $|\lfloor b,c\rfloor|\leq
\|\iota\|\|b\|\|c\|$, hence the functions $\lfloor \cdot,c\rfloor$
and $\lfloor b,\cdot\rfloor$ are continuous, for all $b,c\in B$.
The quadratic form $q$ is also a continuous function in this
setting, since $|q(b)-q(c)|=\frac{1}{2}|\lfloor b,b\rfloor-\lfloor
c,c\rfloor|=\frac{1}{2}|\lfloor b-c,b+c\rfloor|\leq\frac{1}{2}
\|\iota\|\|b-c\|\|b+c\|$ for all $b,c\in B$.

(ii) For a function $f:B\rightarrow \B$ we have $f^@(c)=\sup_{b\in
B}\{\langle b,\iota(c)\rangle-f(b)\}=f^*(\iota(c))$, that is
$f^@=f^*\circ\iota$ on $B$.
\end{remark}

\begin{example}\label{ex-b-ssd-sp} (a) The SSD spaces considered in
Example \ref{ex-ssd-sp} (a)-(c) are Banach SSD spaces (Simons
spaces) (see \cite[Remark 2.2]{simons-ssd}).

(b) Consider again the framework of Example \ref{ex-ssd-sp} (d),
that is $X$ is a nonzero Banach space and $B:=X\times X^*$. The
canonical embedding of $X$ into $X^{**}$ is defined by $ \
$$\widehat{}:X\rightarrow X^{**}$, $\langle
x^*,\widehat{x}\rangle:=\langle x,x^*\rangle$ for all $x\in X$ and
$ x^*\in X^*$. The dual of $B$ (with respect to the norm topology)
is $X^*\times X^{**}$ under the pairing
$$\langle b,c^*\rangle=\langle x,y^*\rangle+\langle x^*,y^{**}\rangle\mbox{ for all }b=(x,x^*)\in B,c^*=(y^*,y^{**})\in B^*.$$
Then $X\times X^*$ is a Banach SSD space (Simons space), where
$\iota:X\times X^*\rightarrow X^*\times X^{**}$,
$\iota(x,x^*):=(x^*,\widehat{x})$ for all $(x,x^*)\in X\times
X^*$. \end{example}

For $E:\R_+\rightrightarrows B$ we define
$\overline{E}:\R_+\rightrightarrows B$ by
$\overline{E}(\varepsilon):=\{b\in B:(\varepsilon,b)\in
\overline{G(E)}\}$. $E$ is said to be closed if $E=\overline{E}$.
One can see that $E$ is closed if and only if $G(E)$ is closed.
For $A\subseteq B$, consider also the subfamily $\E_c(A)=\{E\in
\E(A):E\mbox{ is closed}\}$.

\begin{proposition}\label{prop-enl-cl} Let $B$ be a Banach SSD
space (Simons space) and $A\subseteq B$ be a maximally
$q$-positive set. Then
\begin{enumerate}\item[(i)] If $E\in\E(A)$ then $\overline{E}\in
\E_c(A)$.

\item[(ii)] If $E\in\E_c(A)$ then $E(\varepsilon)$ is closed, for
all $\varepsilon\geq 0$.

\item[(iii)] $\overline{E_A}$ and $E^A$ are, respectively, the
smallest and the biggest elements in $\E_c(A)$, with respect to
the partial ordering inclusion relation of the graphs, that is
$G(\overline{E_A})\subseteq G(E)\subseteq G(E^A)$ for all $E\in
\E_c(A)$.
\end{enumerate}\end{proposition}

\begin{proof} (i) Let be $E\in\E(A)$. One can note that the continuity of the function
$q$ implies that if $E$ satisfies the transportation formula, then
$\overline{E}$ satisfies this formula, too. Further, if $E$ is
nondecreasing, then $\overline{E}$ is also nondecreasing. Hence
the first assertion follows.

(ii) The second statement of the proposition is a consequence of
the fact that $E$ is closed if and only if $G(E)$ is closed.

(iii) Employing once more the continuity of the function $q$ we
get that $G(E^A)$ is closed. Combining Proposition \ref{sm-big-E}
and Proposition \ref{prop-enl-cl} (i) we obtain
$\overline{E_A},E^A\in \E_c(A)$. The proof of the minimality,
respectively, maximality of these elements presents no difficulty.
\end{proof}

In the following we establish a one-to-one correspondence between
$\E_c(A)$ and a family of convex functions with certain
properties. This is done be extending the techniques used in
\cite[Section 3]{bu-sv-02} to Banach SSD spaces (Simons spaces).

Consider $B$ a Banach SSD space (Simons space). To $A\subseteq
B\times \R$ we associate the so-called \emph{lower envelope of
$A$} (cf. \cite{avriel}), defined as $\lev_A:B\rightarrow \B$,
$\lev_A(b)=\inf\{r\in\R:(b,r)\in A\}$. Obviously, $A\subseteq
\epi(\lev_A)$. If, additionally, $A$ is closed and has an
\emph{epigraphical structure}, that is $(b,r_1)\in A\Rightarrow
(b,r_2)\in A$ for all $r_2\in [r_1,+\infty)$, then
$A=\epi(\lev_A)$.

Let us consider now $E:\R_+\rightrightarrows B$ and define
$\lambda_E:B\rightarrow \B$, $\lambda_E(b)=\inf\{\varepsilon\geq
0:b\in E(\varepsilon)\}$. It is easy to observe that
$\lambda_E(b)=\inf\{r\in\R:(b,r)\in G(E^{-1})\}$, where $E^{-1} :B
\rightrightarrows \R_+$ is the \emph{inverse of the multifunction
$E$}. One has $G(E^{-1})=\{(b,\varepsilon):(\varepsilon,b)\in
G(E)\}$. Hence $\lambda_E$ is the lower envelope of $G(E^{-1})$.
We have $G(E^{-1})\subseteq\epi(\lambda_E)$. If $E$ is closed and
nondecreasing, then $G(E^{-1})$ is closed and has an epigraphical
structure, so in this case $G(E^{-1})=\epi(\lambda_E)$. As in
\cite[Proposition 3.1]{bu-sv-02} we obtain the following result.

\begin{proposition}\label{le} Let $B$ be a Banach SSD space (Simons space) and $E:\R_+\rightrightarrows
B$ a multifunction which is closed and nondecreasing. Then:
\begin{enumerate} \item[(i)] $G(E^{-1})=\epi(\lambda_E)$;

\item[(ii)] $\lambda_E$ is lower semicontinuous;

\item[(iii)] $\lambda_E\geq 0$;

\item[(iv)] $E(\varepsilon)=\{b\in B:
\lambda_E(b)\leq\varepsilon\}$ for all $\varepsilon\geq 0$.
\end{enumerate}

Moreover, $\lambda_E$ is the only function from $B$ to $\B$
satisfying (iii) and (iv).
\end{proposition}

Given $E:\R_+\rightrightarrows B$, we define the function
$\Lambda_E:B\rightarrow\B$, $\Lambda_E:=\lambda_E+q$. Let us
notice that $\Lambda_E$ is the lower envelope of $\Psi(G(E^{-1}))$
(the function $\Psi$ was defined in Proposition
\ref{transp-f2n-conv}) and
$\epi(\Lambda_E)=\Psi(\epi(\lambda_E))$. From these observations,
Proposition \ref{le} (i) and Proposition \ref{transp-f2n-conv} we
obtain the following result.

\begin{corollary}\label{tr-conv-Lambda} Let $B$ be  Banach SSD
space (Simons space) and $E:\R_+\rightrightarrows B$ a closed and
nondecreasing enlargement of the maximally $q$-positive set
$A\subseteq B$. Then $E\in \E(A)$ if and only if $\Lambda_E$ is
convex.
\end{corollary}

\begin{proposition}\label{pr} Let $B$ be a Banach SSD space (Simons space), $A\subseteq B$ a
maximally $q$-positive set and $E\in \E_c(A)$. Then $\Lambda_E$ is
convex, lower semicontinuous, $\Lambda_E\geq q$ on $B$ and
$A\subseteq  {\cal P}(\Lambda_E)$.\end{proposition}

\begin{proof} The first three assertions follow from Corollary
\ref{tr-conv-Lambda} and Proposition \ref{le} (ii) and (iii). Take
an arbitrary $b\in A$. Since $E$ is an enlargement of $A$ we get
$b\in E(0)$, hence $\lambda_E(b)=0$ and the conclusion follows.
\end{proof}

\noindent To every maximally $q$-positive set we introduce the
following family of convex functions.

\begin{definition}\label{def-repr} Let $B$ be a Banach SSD space (Simons space) and $A\subseteq B$ be a
maximally $q$-positive set. We define ${\cal H}(A)$ as the family
of convex lower semicontinuous functions $h:B\rightarrow\B$ such
that $$h\geq q\mbox{ on }B\mbox{ and }A\subseteq {\cal
P}(h).$$\end{definition}

\begin{remark}\label{one-to-one} Combining Proposition \ref{pr} and Proposition
\ref{le} (i) we obtain that the map $E\mapsto\Lambda_E$ is
one-to-one from $\E_c(A)$ to ${\cal H}(A)$.\end{remark}

For $h\in {\cal H}(A)$ we define the multifunction
$A_h:\R_+\rightrightarrows B$, $$A_h(\varepsilon):=\{b\in
B:h(b)\leq\varepsilon+q(b)\}\mbox{ for all }\varepsilon\geq 0.$$

\begin{proposition}\label{A_h} Let $B$ be a Banach SSD space (Simons space) and $A\subseteq B$ be a
maximally $q$-positive set. If $h\in{\cal H}(A)$, then
$A_h\in\E_c(A)$ and $\Lambda_{A_h}=h$.\end{proposition}

\begin{proof} Take an arbitrary $h\in{\cal H}(A)$. The properties
of the function $h$ imply that $A_h$ is a closed enlargement of
$A$. Obviously $A_h$ is nondecreasing. Trivially,
$A_h(\varepsilon)=\{b\in B:l(b)\leq\varepsilon\}$, where
$l:B\rightarrow\B$, $l:=h-q$.  By Proposition \ref{le} we get
$\lambda_{A_h}=l$, implying $\Lambda_{A_h}=h$. The convexity of
$h$ and Corollary \ref{tr-conv-Lambda} guarantee that
$A_h\in\E_c(A)$. \end{proof}

As a consequence of the above results we obtain a bijection
between the family of closed enlargements (which satisfy condition
(r1)-(r3) from Definition \ref{fam-enl}) associated to a maximally
positive set and the family of convex functions introduced in
Definition \ref{def-repr}.

\begin{theorem}\label{bijection} Let $B$ be a Banach SSD space (Simons space) and $A\subseteq B$ be a
maximally $q$-positive set. The map $$\E_c(A)\rightarrow {\cal
H}(A),$$$$E\mapsto\Lambda_E$$ is a bijection, with inverse given
by $${\cal H}(A)\rightarrow \E_c(A),$$$$h\mapsto A_h.$$ Moreover,
$A_{\Lambda_E}=E$ for all $E\in \E_c(A)$ and $\Lambda_{A_h}=h$ for
all $h\in {\cal H}(A)$.
\end{theorem}

The following corollary shows that there exists a closed
connection between an element $h\in {\cal H}(A)$ and the maximally
positive set $A$.

\begin{corollary}\label{repr} Let $B$ be a Banach SSD space (Simons space) and $A\subseteq B$ be a
maximally $q$-positive set. Take $h\in{\cal H}(A)$. Then $A={\cal
P}(h)$.\end{corollary}

\begin{proof} We have $A\subseteq {\cal
P}(h)$ by the definition of ${\cal H}(A)$. Take an arbitrary $b\in
{\cal P}(h)$. Define $E:=A_h$. Then $b\in E(0)$. Applying Theorem
\ref{bijection} we get $E\in \E_c(A)$. Further, by Lemma
\ref{gen-prop-enl} we have $E(0)=A$, hence $b\in A$ and the proof
is complete. \end{proof}

\begin{remark} In what follows, we call an arbitrary element $h$ of ${\cal
H}(A)$ a \textbf{representative functions} of $A$. The word
"representative" is justified by Corollary \ref{repr}. Since for
$A$ a $q$-positive set, we have $A\neq\emptyset$ (see Definition
\ref{ssd} (ii)), every representative function of $A$ is proper.
\end{remark}

\begin{corollary}\label{ineq} Let $B$ be a Banach SSD space (Simons space) and $A\subseteq B$ be a
maximally $q$-positive set. Take $E\in \E_c(A)$ and $b^1\in
E(\varepsilon_1),b^2\in E(\varepsilon_2)$, where
$\varepsilon_1,\varepsilon_2\geq 0$ are arbitrary. Then
$$q(b^1-b^2)\geq - ({\sqrt{\varepsilon_1}}+{\sqrt{\varepsilon_2}})^2.$$\end{corollary}

\begin{proof} By Theorem \ref{bijection}, there exists a representative function $h\in{\cal
H}(A)$ such that $E=A_h$. By using the definition of $A_h$ and
applying \cite[Lemma 1.6]{simons-ssd} we obtain
$$-q(b^1-b^2)\leq\big[\sqrt{(h-q)(b^1)}+\sqrt{(h-q)(b^2)}\big]^2\leq
(\sqrt{\varepsilon_1}+\sqrt{\varepsilon_2})^2$$ and the proof is
complete. \end{proof}

\begin{remark}\label{ineq-bu-sv} The above lower bound is
established in \cite[Corollary 3.12]{bu-sv-99}, where $B$ is taken
as in Example \ref{ex-b-ssd-sp} (b), for the enlargement $B^A$
(see Example \ref{E_A}). Here we generalize this result to Banach
SSD spaces (Simons spaces) and to an arbitrary $E\in
\E_c(A)$.\end{remark}

In what follows we investigate the properties of the functions
$\Lambda_{E^A}, \Lambda_{\overline{E_A}}$ and rediscover in this
way the functions introduced and studied by S. Simons in
\cite{simons-ssd, simons-ssd-jca, simons} (see Proposition
\ref{furth-prop-fitz} (iii) from below).

\begin{corollary}\label{repr-min-max} Let $B$ be a Banach SSD space (Simons space) and $A\subseteq B$ be a
maximally $q$-positive set.\begin{enumerate} \item[(i)] The
functions $\Lambda_{E^A}, \Lambda_{\overline{E_A}}\in {\cal H}(A)$
and are respectively the minimum and the maximum of this family,
that is \begin{equation}\label{ineq-fam1}\Lambda_{E^A}\leq h\leq
\Lambda_{\overline{E_A}}\mbox{ for all }h\in {\cal
H}(A).\end{equation}

\item[(ii)] Conversely, if $h:B\rightarrow \B$ is a convex lower
semicontinuous function such that
\begin{equation}\label{ineq-fam2}\Lambda_{E^A}\leq h\leq
\Lambda_{\overline{E_A}},\end{equation} then $h\in {\cal H}(A).$

\item[(iii)] It holds ${\cal H}(A)=\{h:B\rightarrow\B: h \mbox{
convex, lower semicontinuous and }\Lambda_{E^A}\leq h\leq
\Lambda_{\overline{E_A}}\}$.\end{enumerate}
\end{corollary}

\begin{proof} (i) This follows immediately from Theorem
\ref{bijection} and Proposition \ref{prop-enl-cl}.

(ii) If $h:B\rightarrow \B$ is a convex lower semicontinuous
function satisfying \eqref{ineq-fam2}, then (since
$\Lambda_{E^A}\in {\cal H}(A)$)
\begin{equation}\label{ineq-fam3}h\geq \Lambda_{E^A}\geq q\mbox{
on }B.\end{equation} Further, for $b\in A$ we obtain (employing
that $\Lambda_{\overline{E_A}}\in {\cal H}(A)$) that $h(b)\leq
\Lambda_{\overline{E_A}}(b)=q(b)$. In view of \eqref{ineq-fam3} it
follows that $b\in{\cal P}(h)$, hence $h\in {\cal H}(A)$.

(iii) This characterization of ${\cal H}(A)$ is a direct
consequence of (i) and (ii). \end{proof}

\begin{definition}\label{fitz-f} (cf. \cite{simons-ssd}) Let $B$ be a Banach SSD space (Simons space) and
$A\subseteq B$ be a $q$-positive set. We define the function
$\Theta_A: B^*\rightarrow \B$, $$\Theta_A(b^*):=\sup_{a\in
A}[\langle a,b^*\rangle-q(a)]\mbox{ for all }b^*\in B^*.$$ We
define the function $\Phi_A:B\rightarrow\B$,
$$\Phi_A:=\Theta_A\circ\iota$$ and also the function
$^*\Theta_A:B\rightarrow\B$, $$^*\Theta_A(c):=\sup_{b^*\in
B^*}[\langle c,b^*\rangle-\Theta_A(b^*)]\mbox{ for all }c\in B.$$
\end{definition}

We denote by $a\vee b$ the maximum value between $a,b\in\B$. The
following properties of the functions defined above appear in
\cite[Lemma 2.13 and Theorem 2.16]{simons-ssd}. The property (vii)
is a direct consequence of (i)-(vi).

\begin{lemma}\label{propr-fitz} Let $B$ be a Banach SSD space (Simons space) and
$A\subseteq B$ a $q$-positive set. Then:

(i) For all $b\in B$, $\Phi_A(b)=\sup_{a\in A}[\lfloor
a,b\rfloor-q(a)]=q(b)-\inf_{c\in A} q(b-c)$.

(ii) $\Phi_A$ is proper, convex, lower semicontinuous and
$A\subseteq{\cal P}(\Phi_A)$.

(iii) $(^*\Theta_A)^*=\Theta_A$ and $(^*\Theta_A)^@=\Phi_A$.

(iv) $^*\Theta_A$ is proper, convex, lower semicontinuous,
$^*\Theta_A\geq \Phi_A^{@}\geq \Phi_A\vee q$ on $B$ and
$$^*\Theta_A=\Phi_A^@=q\mbox{ on }A.$$

(v) $^*\Theta_A=\sup\{h:B\rightarrow\B: h\mbox{ proper, convex,
lower semicontinuous, }h\leq q\mbox{ on }A\}$.

If, additionally, $A$ is maximally $q$-positive, then:

(vi) $^*\Theta_A\geq \Phi_A^{@}\geq \Phi_A\geq q$ on $B$ and
$A={\cal P}(^*\Theta_A)={\cal P}(\Phi_A^@)={\cal P}(\Phi_A)$.

(vii) $^*\Theta_A, \Phi_A^{@},\Phi_A\in {\cal H}(A)$.
\end{lemma}

Next we give other characterizations of the function $^*\Theta_A$
and establish the connection between $\Lambda_{E^A},
\Lambda_{\overline{E_A}}$ and $\Phi_A, {^*\Theta_A}$,
respectively.

\begin{proposition}\label{furth-prop-fitz} Let $B$ be a Banach SSD space (Simons space) and
$A\subseteq B$ a $q$-positive set. Then:

(i) $^*\Theta_A=\sup\{h:B\rightarrow\B: h\mbox{ proper, convex,
lower semicontinuous, }h\geq q\mbox{ on }B \ \mbox{and}$ $A
\subseteq {\cal P}(h)\}$.

(ii) $^*\Theta_A=\cl\co(q+\delta_A)$.

If, additionally, $A$ is maximally $q$-positive, then:

(iii) $\Lambda_{E^A}=\Phi_A$ and
$\Lambda_{\overline{E_A}}={^*\Theta_A}$.

(iv) If $h:B\rightarrow\B$ is a function such that $h\in{\cal
H}(A)$, then $h^@\in {\cal H}(A)$.
\end{proposition}

\begin{proof} (i) We have $\{h:B\rightarrow\B: h\mbox{ proper, convex, lower semicontinuous,
}h\geq q\mbox{ on }B$ $\mbox{ and }A\subseteq{\cal
P}(h)\}\subseteq \{h:B\rightarrow\B: h\mbox{ proper, convex, lower
semicontinuous, }h\leq q\mbox{ on }A\}$ hence from Lemma
\ref{propr-fitz} (v) we get
$$\sup\{h:B\rightarrow\B: h\mbox{ proper, convex, lower semicontinuous,
}$$$$h\geq q\mbox{ on }B\mbox{ and }A\subseteq{\cal P}(h)\}\leq
{^*\Theta_A}.$$ On the other hand, $^*\Theta_A\in
\{h:B\rightarrow\B: h\mbox{ proper, convex, lower semicontinuous,
}h\geq q\mbox{ on }B\mbox{ and }A\subseteq{\cal P}(h)\}$ (see
Lemma \ref{propr-fitz} (iv)), thus
$$\sup\{h:B\rightarrow\B: h\mbox{ proper, convex, lower semicontinuous, }$$$$h\geq
q\mbox{ on }B\mbox{ and }A\subseteq{\cal P}(h)\}\geq
{^*\Theta_A}$$ and the equality follows.

(ii) Since $^*\Theta_A\leq q$ on $A$ we have $^*\Theta_A\leq
q+\delta_A$ on $B$, hence
\begin{equation}\label{ineq-penot-f}^*\Theta_A\leq
\cl\co(q+\delta_A)\leq q+\delta_A.\end{equation} The above
inequality shows that $\cl\co(q+\delta_A)$ is a proper, convex,
lower semicontinuous function such that $\cl\co(q+\delta_A)\leq q$
on $A$. Applying Lemma \ref{propr-fitz} (v) we obtain
$^*\Theta_A\geq\cl\co(q+\delta_A)$, which combined with
\eqref{ineq-penot-f} delivers the desired result.

(iii) From Lemma \ref{propr-fitz} (i) and the definition of $E^A$
we obtain $$b\in E^A(\varepsilon)\Leftrightarrow
q(b-c)\geq-\varepsilon\mbox{ for all }c\in A\Leftrightarrow
\inf_{c\in A} q(b-c)\geq-\varepsilon$$$$\Leftrightarrow
q(b)-\Phi_A(b)\geq-\varepsilon\Leftrightarrow
\Phi_A(b)\leq\varepsilon+q(b).$$ This is nothing else than
$E^A=A_{\Phi_A}$. Theorem \ref{bijection} implies that
$\Lambda_{E^A}=\Lambda_{A_{\Phi_A}}=\Phi_A$.

The equality $\Lambda_{\overline{E_A}}={^*\Theta_A}$ follows from
(i) and Corollary \ref{repr-min-max}.

(iv) From (iii), Corollary \ref{repr-min-max} and \cite[Theorem
2.15 (b)]{simons-ssd} we get $h^@\geq q$ on $B$ and ${\cal
P}(h)={\cal P}(h^@)=A$. The function $h^@$ is proper, convex and
the lower semicontinuity follows from the definition of $h^@$ and
Remark \ref{cont-q}, hence $h^@\in {\cal H}(A)$.
\end{proof}

\begin{remark}\label{gen-maps} Proposition \ref{furth-prop-fitz} (iv) is a
generalization of \cite[Theorem 5.3]{bu-sv-02} to Banach SSD
spaces (Simons spaces).\end{remark}

\begin{remark}\label{not-id1} In general, the functions $^*\Theta_A$ and
$\Phi_A^@$ are not identical. An example in this sense was given
by C. Z\u alinescu (see \cite[Remark 2.14]{simons-ssd}). However,
this example is restrictive, since it is given in a Banach space
$B$ such that $\lfloor\cdot,\cdot\rfloor=0$ on $B\times B$. An
alternative example, originally due to M.D. Voisei and C. Z\u
alinescu, can be given to prove this fact (see Example
\ref{not-id2} below).\end{remark}

\noindent Before we present this example, we need the following
remark.

\begin{remark}\label{x-x-star} Consider again the particular setting of Example
\ref{ex-ssd-sp}(d) and Example \ref{ex-b-ssd-sp} (b), namely when
$B=X\times X^*$, where $X$ is a nonzero Banach space. Let $A$ be a
nonempty monotone subset of $X\times X^*$. In this case
$q(x,x^*)=\langle x,x^*\rangle$ for all $(x,x^*)\in X\times X^*$
and the function $\Theta_A:X^*\times X^{**}\rightarrow\B$ is
defined by
$$\Theta_A(x^*,x^{**})=\sup_{(s,s^*)\in A}[\langle
s,x^*\rangle+\langle s^*,x^{**}\rangle-\langle s,s^*\rangle]\mbox{
for all }(x^*,x^{**})\in X^*\times X^{**}.$$ The function
$\Phi_A:X\times X^*\rightarrow\B$ has the following form:
$$\Phi_A(x,x^*)=\sup_{(s,s^*)\in A}[\langle s,x^*\rangle+\langle
x,s^*\rangle-\langle s,s^*\rangle]\mbox{ for all }(x,x^*)\in
X\times X^*,$$ that is $\Phi_A=(q+\delta_A)^@$. $\Phi_A$ is the
Fitzpatrick function of $A$. Introduced by S. Fitzpatrick in
\cite{fitz} in 1988 and rediscovered after some years in
\cite{bu-sv-02, legaz-thera2}, it proved to be very important in
the theory of maximal monotone operators, revealing important
connections between convex analysis and monotone operators (see
\cite{bausch, borwein, BCW, BCW-set-val, BGW-max-comp, bu-sv-02,
legaz-sv, penot, penot-zal, sim1, simons, simons-zal, voisei} and
the references therein). Applying the Fenchel-Moreau Theorem we
obtain
\begin{equation}\label{fitz-conj}\Phi_A^@=\cl\nolimits_{s\times w^*}\co(q+\delta_A)\end{equation} (the
closure is taken with respect to the strong-weak$^*$ topology on
$X\times X^*$). The function $^*\Theta_A:X\times X^*\rightarrow\B$
is defined by $$^*\Theta_A(y,y^*)=\sup_{(x^*,x^{**})\in X^*\times
X^{**}}[\langle y,x^*\rangle+\langle
y^*,x^{**}\rangle-\Theta_A(x^*,x^{**})]\mbox{ for all }(y,y^*)\in
X\times X^*.$$
\end{remark}

\begin{example}\label{not-id2} As in \cite[page 5]{voisei-zali-lcs}, consider $E$ a nonreflexive
Banach space, $X:=E^*$ and $A:=\{0\}\times \widehat{E}$, which is
a monotone subset of $X\times X^*$. Let us notice that for $A$ we
have $q+\delta_A=\delta_A$. By applying Proposition
\ref{furth-prop-fitz} (ii) we obtain:
$^*\Theta_A=\cl\co\delta_A=\delta_A$ (the closure is taken with
respect to the strong topology of $X\times X^*$). Further, by
using \eqref{fitz-conj} and the Goldstine Theorem we get
$\Phi_A^@=\cl_{s\times w^*}\co\delta_A=\delta_{\{0\}\times
E^{**}}\neq {^*\Theta_A}$ (the closure is considered with respect
to the strong-weak$^*$ topology of $X\times X^*$).\end{example}

In the last part of the paper we deal with another subfamily of
$\E(A)$, namely the one of closed and \emph{additive}
enlargements. In this way we extend the results from \cite{sv,
bu-sv-02} to SSD Banach spaces (Simons spaces).

\begin{definition}\label{additive} Let $B$ be a Banach SSD space (Simons space). We say that the multifunction $E:\R_+\rightrightarrows
B$ is additive if for all $\varepsilon_1, \varepsilon_2\geq 0$ and
$b^1\in E(\varepsilon_1)$, $b^2\in E(\varepsilon_2)$ one has
$$q(b^1-b^2) \geq -(\varepsilon_1 + \varepsilon_2).$$
In case $A \subseteq B$ is a maximally $q$-positive set we denote
by $\E_{ca}(A) = \{E \in \E_c(A) : E \ \mbox{is additive}\}$.
\end{definition}

We have the following characterization of the set $\E_{ca}(A)$.

\begin{theorem}\label{char-add} Let $B$ be a Banach SSD space (Simons space), $A\subseteq B$ a maximally $q$-positive set and
$E \in \E_c(A)$. Then:
$$E \in \E_{ca}(A) \Leftrightarrow \Lambda_E^@ \leq
\Lambda_E.$$
\end{theorem}

\begin{proof}
Assume first that $E \in \E_{ca}(A)$ and take $b^1, b^2$ two
arbitrary elements in $B$. By Proposition \ref{le} (iii) follows
that $\lambda_E(b^1) \geq 0$ and $\lambda_E(b^2) \geq 0$. We claim
that
$$q(b^1-b^2) \geq -(\lambda_E(b^1) + \lambda_E(b^2)).$$
In case $\lambda_E(b^1) = +\infty$ or $\lambda_E(b^2) = +\infty$
(or both), this fact is obvious. If $\lambda_E(b^1)$ and
$\lambda_E(b^2)$ are finite, the inequality above follows by using
that (cf. Proposition \ref{le} (i)) $(b_1, \lambda_E(b^1))$,
$(b_2, \lambda_E(b^2)) \in \epi(\lambda_E) = G(E^{-1})$ and that
$E$ is additive. Consequently, for all $b^1, b^2 \in B$,
$$\lambda_E(b^1) + q(b^1) \geq \lfloor b^1, b^2 \rfloor - (\lambda_E(b^2) +
q(b^2)) \Leftrightarrow \Lambda_E(b^1) \geq \lfloor b^1, b^2
\rfloor - \Lambda_E(b^2).$$ This means that for all $b^1 \in B$,
$\Lambda_E^@(b^1) \leq \Lambda_E(b^1)$.

Assume now that $\Lambda_E^@ \leq \Lambda_E$ and take arbitrary
$\varepsilon_1, \varepsilon_2\geq 0$ and $b^1\in
E(\varepsilon_1)$, $b^2\in E(\varepsilon_2)$. This means that
$\lambda_E(b^1) \leq \varepsilon_1$ and $\lambda_E(b^2) \leq
\varepsilon_2$. Since $\Lambda_E(b^1) \geq \Lambda_E^@(b^1)$, one
has
$$\lambda_E(b^1) + q(b^1) \geq \lfloor b^1, b^2 \rfloor - (\lambda_E(b^2) +
q(b^2))$$ and from here
$$q(b^1-b^2) \geq -(\lambda_E(b^1) + \lambda_E(b^2)) \geq -(\varepsilon_1 +
\varepsilon_2).$$ This concludes the proof.
\end{proof}

One can use the theorem above for providing an element in the set
$\E_{ca}(A)$ in case $A$ is a maximally $q$-positive set.

\begin{proposition}\label{el-add} Let $B$ be a Banach SSD space (Simons space) and $A\subseteq B$ a maximally $q$-positive set. Then $\overline{E_A} \in  \E_{ca}(A)$.
\end{proposition}

\begin{proof}
By Proposition \ref{furth-prop-fitz} (iii) and Lemma
\ref{propr-fitz} (iii)-(iv) we have  $(\Lambda_{\overline{E_A}})^@
= (^*\Theta_A)^@=\Phi_A \leq {^*\Theta_A} =
\Lambda_{\overline{E_A}}$. Theorem \ref{char-add} guarantees that
$\overline{E_A} \in \E_{ca}(A)$.
\end{proof}

\end{document}